\documentclass{amsart}
\usepackage{amsfonts}
\usepackage{amssymb}
\usepackage{amsmath}

\newtheorem{definition}{Definition}[section]
\newtheorem{theorem}[definition]{Theorem}

\newtheorem{lemma}{Lemma}[section]

\newtheorem{proposition}[definition]{Proposition}
\newtheorem{remark}{Remark}
\newtheorem{example}{Example}
\numberwithin{equation}{section}

\begin{document}
\title[On Statistical Convergence of Metric Valued Sequences]{On Statistical
Convergence of Metric Valued Sequences}
\author{ MEHMET K\"{u}\c{c}\"{u}kaslan*, U\u{g}UR De\u{g}er*, AND OLEKSIY Dovgoshey** }
\address{* Mersin University \newline \hspace*{\parindent}
 Faculty of Science and Literature
 \newline \hspace*{\parindent} Department of Mathematics \newline \hspace*{\parindent} 33343 Mersin \newline \hspace*{\parindent} TURKEY}
\email{mkucukaslan@mersin.edu.tr(mkkaslan@gmail.com)}
\email{udeger@mersin.edu.tr}
\address{** Institute of Applied Mathematics and Mechanics of NASU \newline \hspace*{\parindent}
 R. Luxemburg str. 74
 \newline \hspace*{\parindent} 83114 Donetsk \newline \hspace*{\parindent} UKRAINE}
\email{aleksdov@mail.ru}
\subjclass[2010]{ 40A05,
54A20, 54E35} \keywords{Metric spaces, asymptotic density, $d-$statistical convergence, statistical convergence of subsequences.}

\begin{abstract}
We study the statistical convergence of metric valued sequences and of their
subsequences. The interplay between the statistical and usual convergences  in metric
spaces is also studied.

\end{abstract}

\maketitle

\section{Introduction and Definitions}

Analysis on metric spaces has rapidly developed in present time (see,
\cite{Heinonen}, \cite{Papadopoulos}). This development is usually based on some
generalizations of the differentiability. Ordinary the generalizations of the differentiation
involve linear structure by means of embeddings of metric spaces in a suitable normed
space or by use of geodesics.

A new intrinsic approach to the introduction of the smooth structure for general metric
space was proposed   by O. Martio and O. Dovgoshey in \cite{Dovgoshey1} (see also \cite{DM} and \cite{Dovgoshey3}, \cite{Dor}, \cite{Dovgoshey2}, \cite{DAK}, \cite{DD}). The approach in
\cite{Dovgoshey1} is completely based on the convergence of the metric valued sequences
but it is not apriori clear that the standard convergence is the best possible way to
obtain a smooth structure for arbitrary metric space.

The problem of convergence in different ways of a real (or complex) valued divergent
sequence goes back to the  beginning of nineteenth century. A lot of different
convergence methods were defined (Cesaro, N\"{o}rlund, Weighted Mean, Abel etc.) and
applied to many branches of mathematics. Almost all convergence methods depend on the
algebraic structure of the space. It is clear that metric space does not have the
algebraic structure in general. However, the notion of statistical convergence is easy
to extend for arbitrary metric spaces and this provides a general framework for
summability in such spaces \cite{FK}, \cite{Te}. Thus, the studies of statistical
convergence give a natural foundation for upbuilding of different tangent spaces to
general metric spaces.

The construction of tangent spaces in  \cite{Dovgoshey2}, \cite{DAK},
\cite{DD}, \cite{Dovgoshey1} is based on the following fundamental fact: ``If $(x_n)$ is a convergent
sequence in a metric space, then each subsequence $(x_{n(k)})$ of $(x_n)$ is also
convergent''. Thus the convergence of subsequence $(x_{n(k)})$ does not depend on the
choice of $(x_{n(k)}).$  Unfortunately it is not the case for the statistical convergent
sequences. The applications of the statistical convergence to the infinitesimal geometry
of metric spaces should be based on the complete understanding of the structure of
statistical convergent subsequences. To describe this structure is the main goal of this
paper. Moreover we study some interrelations between the statistical convergence and the
usual one for general metric spaces.

Let us remember the main definitions. Let $(X,d)$ be a metric space. For convenience
denote by $\tilde X$ the set of all sequences of points from $X.$

\begin{definition}
\label{def1}  A  sequence $(x_{n})\in \tilde X$ is called convergent to a point $a \in
X$ if for every $\epsilon >0$ there is $n_{0}=n_{0}(\epsilon)\in
\mathbb{N}
$  \ such that $n>n_{0}$ implies
\begin{equation}
d(x_{n},a)<\epsilon.  \label{equ1}
\end{equation}%

\end{definition}

\begin{definition}
\label{def2+} A  metric valued sequence $\tilde x=(x_{n})\in \tilde X$ is called  $d-$%
statistical convergent to a point $a \in X$ \ if, for every $\epsilon >0,$%
\begin{equation}
\underset{n\rightarrow \infty }{\lim }\frac{1}{n}\left\vert \left\{ k:k\leq n%
\text{ and }d(x_{k},a)\geq \epsilon \right\} \right\vert =0. \label{equ2}
\end{equation}
\end{definition}

In \eqref{equ2} and later  $\left\vert B\right\vert $ denotes the number of elements of
the set $B.$

 The idea
of statistical convergence goes back to Zygmud \cite{Zygmund}. It was formally
introduced by Steinhous \cite{Steinhous} and  Fast \cite{Fast}. In
recent years, it has become an active research for mathematicians \cite%
{Cervanansky}, \cite{Connor}, \cite{Fridy}, \cite{Fridy1}, \cite{Miller}, etc.

\begin{definition}
\label{def3} \cite{Fast}\emph{(Dense subset of $%
\mathbb{N}
$) }A set $K\subseteq \mathbb N$  is called a statistical dense subset of $\mathbb N$ if
\begin{equation*}
\underset{n\rightarrow \infty }{\lim }\frac{1}{n}\left\vert K(n)\right\vert =1
\end{equation*}%
where $K(n):=\left\{ k\in K: k \leq n\right\} .$
\end{definition}

It may be proved that the intersection of two dense subsets of natural numbers is dense.
Moreover it is clear that the supersets of dense sets are also dense. Hence the family
of all dense subsets of $\mathbb N$ is a filter on $\mathbb N.$ Theorem~\ref{theo3+},
given at the end of the next section, implies, in particular, that the $d$ --
statistical convergence is simply the convergence in $(X,d)$ with respect to this
filter.

\begin{remark}
If \ $K_{1}$ is a statistical dense subset of $\mathbb N,$ $K_{2}\subseteq \mathbb N$
and $\underset{n\rightarrow \infty }{\lim }\frac{\left\vert K_{1}(n)\right\vert }{%
\left\vert K_{2}(n)\right\vert }=1,$ then $K_{2}$ is also a
statistical dense subset of $\mathbb N$.
\end{remark}
\begin{definition}
\label{def4}\emph{(Dense subsequence)} If $(n(k))$ is an infinite, strictly
increasing sequence of natural numbers and
$\tilde{x}=(x_n)\in\tilde{X}$, define
\begin{equation}\label{K}
\tilde x'=(x_{n(k)}) \quad and \quad K_{\tilde x'}=\left\{ n(k):k\in
\mathbb{N}
\right\} .
\end{equation}
The subsequence $\tilde x'$ of $\tilde x$ is called a dense subsequence of $\tilde x$ if $K_{\tilde
x'}$ is a dense subset of $\mathbb N$. \end{definition}

In our next definition we introduce an equivalence relation on the set $\tilde
 X.$
 \begin{definition}\label{def6} The sequences $\tilde x=(x_n)\in \tilde X$ and $\tilde y=(y_n)\in\tilde X$ are called statistical equivalent if there is a statistical dense subset $M$ of $\mathbb N$ such that $x_n=y_n$ for each $n\in M.$ \end{definition}

We write $\tilde x \asymp \tilde y$ if $\tilde x$ and $\tilde y$ are statistical
equivalent.
\section{Convergent sequences and statistical convergent ones}

In this section, some basic results on $d-$statistical convergence will be given for an
arbitrary metric space. In particular, it is shown that there is a one-to-one
correspondence between metrizable topologies on $X$ and the subsets of $\tilde X$
consisting of statistical convergent sequences determined by some metric compatible with the topologies.

The following result is well known.

\begin{proposition}
\bigskip \label{Prop1} Let $(X,d)$ be a metric space, $(x_{n})\in \tilde X$ and $a\in X$. If $(x_n)$ is convergent to $a,$ then $(x_n)$ is $d-$statistical
convergent to $a.$

\end{proposition}

The converse of Proposition~\ref{Prop1} is not true in general. \begin{example} Assume
that $x,y\in X$
 are distinct ($x\neq y$) and
 define the following sequence%
\begin{equation*}
x_{n}:=\left\{
\begin{array}{cccc}
x & \mbox{if} & n=k^{2} & \mbox{for some } k\in\mathbb{N}\\
y & \mbox{if} & n\neq k^{2} & \mbox{for all } k\in\mathbb{N}%
\end{array}%
\right.
\end{equation*}%
that is,
\begin{equation}
(x_{n})=(\overset{1^{2}}{\widehat{x}},y,y,\overset{2^{2}}{\widehat{x}}%
,y,y,y,y,\overset{3^{2}}{\widehat{x}},y,y,y,y,y,y,\overset{4^{2}}{\widehat{x}%
},...).  \label{equ3+}
\end{equation}

This sequence is not a Cauchy sequence because $d(x_{n^{2}},x_{n^{2}+1})=d(x,y)>0$ for
every $n\in\mathbb N.$ Consequently $(x_n)$ is not convergent.
Let us show that the sequence $\tilde x=(x_{n})$ is $d-$statistical convergent to $y.$
Denote
\begin{equation*}
A(n,\epsilon ):=\left\{ m:m\leq n\text{ and }d(x_{m},y)\geq \epsilon \right\}
\end{equation*}%
for every $\epsilon >0$ and $n\in\mathbb N$. We must prove that \begin{equation}
\label{equ2.4}\lim_{n\to\infty}\frac{\left\vert A(n, \epsilon) \right \vert}{n}=0
\end{equation} for each $\epsilon >0.$
Since $ A(n,\epsilon_1)\supseteq A(n,\epsilon_2)$ for $\epsilon_1 \le \epsilon_2,$ it is sufficient to take $\epsilon =$ $= d(x,y).$ In this case a simple calculation shows that
$\left\vert A(n,\epsilon )\right\vert \leq  \sqrt{n}%
.$ Limit relation \eqref{equ2.4} follows. \end{example}

For singleton sets the converse of Proposition \ref{Prop1} is true.

\begin{theorem}
\label{theo1+} Let $(X,d)$ be a nonempty metric space. The following two
statements are equivalent.
\begin{enumerate}
 \item[\rm(i)]\textit{The set of all convergent sequences $\tilde x \in \tilde X$ is the same as the
set of all  d -- statistical~convergent~sequences~$\tilde x \in \tilde X.$}
\item[\rm(ii)]\textit{ The set $X$ is a singleton.}
\end{enumerate}
\end{theorem}

\begin{proof}
 Let us assume $X:=\left\{ x\right\} .$ In this case,
$\tilde X$ contains only the constant sequence
$(x,x,x,...)$ which is convergent and $d$ -- statistical convergent. Therefore,
the set of all convergent sequences $(x_n)\in \tilde X$ coincides
the set of all $d$ -- statistical convergent sequences $(x_n)\in
\tilde X.$ The implication $(ii)\Rightarrow (i)$ is proved.

The implication $(i)\Rightarrow (ii)$ follows from Proposition~\ref{Prop1} and Example
1.\end{proof}

In accordance with Theorem~\ref{theo1+} for every non degenerate metric space $(X,d)$
there are $d$ -- statistical convergent sequences $\tilde x \in \tilde X$ which are
divergent. Nevertheless we have the following result.

\begin{theorem}\label{theo2+} Let $(X,d_1)$ and $(X,d_2)$ be two metric spaces with the same
underlining set $X.$ Then the following three statement are equivalent.
\begin{enumerate}
\item[\rm(i)]\textit{The set of $d_1$ -- statistical convergent sequences coincides with the set of \, \, $d_2$ -- statistical convergent sequences. }
\item[\rm(ii)]\textit{The set of sequences which are convergent in the space $(X,d_1)$ coincides with the set of sequences which are convergent in the space $(X,d_2).$}
\item[\rm(iii)]\textit{The metrics $d_1$ and $d_2$ induce one and the same topology on $X.$}
\end{enumerate}
\begin{proof}
If the metric spaces $(X,d_1)$ and $(X,d_2)$ have the common topology, then for every
$a\in X$ and every $\epsilon >0$ there is $\delta = \delta (\epsilon)>0$ such that
\begin{equation*}
\{x\in X: d_{1}(x,a)<\epsilon\}\supseteq \{x\in X: d_{2}(x,a)<\delta\}.
\end{equation*}
This inclusion implies the inequality
\begin{equation*}
\left\vert\{k\in \mathbb N: k\le n \, and \,   d_{1}(x_{k},a)<\epsilon\}\right \vert \ge
\left\vert\{k\in \mathbb N: k\le n \, and \,   d_{2}(x_{k},a)<\delta\}\right \vert
\end{equation*} for every $\tilde x=(x_k)\in\tilde X$ and $n\in\mathbb N.$ If $\tilde
x$ is  $d_2$ -- statistical convergent to $a,$ then using the last inequality we obtain
\begin{equation*}
1\ge \liminf_{n\to\infty}\frac{\left\vert\{k\in \mathbb N: k\le n \, and \,
d_{1}(x_{k},a)<\epsilon\}\right \vert}{n}\end{equation*}
\begin{equation*}
\ge \liminf_{n\to\infty}\frac{\left\vert\{k\in \mathbb N: k\le n \, and \,
d_{2}(x_{k},a)<\delta\}\right \vert}{n}
\end{equation*}
\begin{equation*}
=\lim_{n\to\infty}\frac{\left\vert\{k\in \mathbb N: k\le n \, and \,
d_{2}(x_{k},a)<\delta\}\right \vert}{n}=1.
\end{equation*}
Consequently
\begin{equation}\label{eq2.4}
\liminf_{n\to\infty}\frac{\left\vert\{k\in \mathbb N: k\le n \, and \,
d_{1}(x_{k},a)<\epsilon\}\right \vert}{n}=1
\end{equation} for every $\epsilon >0.$
Since
$$\mathop{\limsup}\limits_{n\to\infty}\frac{\left\vert\{k\in
\mathbb N: k\le n \, and \, d_{1}(x_{k},a)<\epsilon\}\right
\vert}{n}\le 1,
$$
equality \eqref{eq2.4} implies
\begin{equation*}
\lim_{n\to\infty}\frac{\left\vert\{k\in \mathbb N: k\le n \, and \,
d_{1}(x_{k},a)<\epsilon\}\right \vert}{n}=1
\end{equation*} for every $\epsilon >0.$ Thus if $(x_k)$ is $d_2$ -- statistical convergent to
$a,$ then $(x_k)$ is $d_1$ -- statistical convergent to $a.$ The converse implication can be obtained similarly. So the set of $d_1$ -- statistical convergent
sequences and the set $d_2$ -- statistical convergent sequences are the same if $(X,
d_1)$ and $(X, d_2)$ have the common topology. The implication $(iii)\Rightarrow (i)$
follows.

Suppose now that the topologies induced by $d_1$ and $d_2$ are distinct. Then there
exist a point $a\in X$ and $\epsilon_{0}>0$ such that either
\begin{equation}\label{eq2.5}
\{x\in X: d_{1}(x,a) < \epsilon_{0}\} \nsupseteq \{x\in X: d_{2}(x,a) < \delta\}
\end{equation} for every $\delta >0$ or
\begin{equation*}
\{x\in X: d_{2}(x,a) < \epsilon_{0}\} \nsupseteq \{x\in X: d_{1}(x,a) < \delta\}
\end{equation*} for every $\delta >0.$
We assume, without loss of generality, that \eqref{eq2.5} holds for every $\delta >0$. Then there is a
sequence $\tilde x=(x_n)$ such that
\begin{equation}\label{eq2.6}
 d_{2}(x_n,a) < \frac{1}{n} \qquad and \qquad d_{1}(x_n,a)\ge\epsilon_{0}
\end{equation}for each $n\in\mathbb N.$ Let us define a new sequence $\tilde
y=(y_n)\in\tilde X$ by the rule
\begin{equation*}
y_n:=\begin{cases}
         x_n & \mbox{if} $ $ n $ $ \mbox{is odd} \\
         a   & \mbox{if} $ $ n $ $ \mbox{is even}. \\
         \end{cases}
\end{equation*}This definition and \eqref{eq2.6} imply the equality
\begin{equation}\label{eq2.7}
 \lim_{n\to\infty}\frac{\left\vert\{k\in \mathbb N:
d_{1}(y_{k},a)\ge\epsilon_{0} \, and \, k\le n\}\right \vert}{n}=\frac{1}{2}.
\end{equation} It is clear that the sequence $\tilde y$ is $d_2$ -- statistical convergent to
$a.$ If statement $(i)$ holds, then $\tilde y$ is also $d_1$ --
statistical convergent. Using Theorem~\ref{theo2} (the proof of this theorem does not depend on
Theorem~\ref{theo2+}) we obtain that $\tilde y$ is $d_1$ --
statistically convergent to the same $a.$ Consequently we have
\begin{equation*}
 \lim_{n\to\infty}\frac{\left\vert\{k\in \mathbb N:
d_{1}(y_{k},a)\ge\epsilon_{0} \, and \, k\le n\}\right \vert}{n}=0,
\end{equation*} contrary to \eqref{eq2.7} Thus the implication $(i)\Rightarrow (iii)$
holds and we obtain the equivalence $(iii)\Leftrightarrow (i).$

The equivalence $(iii)\Leftrightarrow (ii)$ can be obtained similarly and we omit the
proof here.
\end{proof}
\end{theorem}

In the rest of this section we prove the following ``weak'' converse of
Proposition~\ref{Prop1}.

\begin{theorem}\label{theo3+}
Let $(X,d)$ be a metric space, $a\in X$ and let $\tilde x=(x_n)\in\tilde X$ be $d$~--~
statistically convergent to $a$. There is $\tilde y=(y_n)\in\tilde X$ such that
$\tilde y\asymp \tilde x$ and $\tilde y$ is convergent to $a.$
\end{theorem}

If $X=\mathbb R$ and $d(x,y)=\vert x - y \vert$ for all $x,y\in X,$ then this result is
known. (See, for example, Theorem A in \cite{MS} or \cite{S}, Lemma 1.1.).

The next simple lemma gives us a tool for the reduction of some questions on the $d$ --
statistical convergence in metric spaces to the case of the statistical convergence in
$\mathbb R.$

\begin{lemma}\label{lem1+}
Let $(X,d)$ be a metric space, $a\in X$ and $\tilde x=(x_n)\in\tilde X.$ Then $\tilde x$
is $d~-~$statistical convergent to $a$ in $X$ if and only if the sequence $(d(x_n,a))$
is statistical convergent to $0$ in $\mathbb R.$
\end{lemma}
The proof follows directly from the definitions.

\vspace{5mm}\emph{Proof of Theorem~\ref{theo3+}.} By Lemma~\ref{lem1+} the sequence $(d(x_n,a))$ is
statistically convergent to $0.$ As has been stated above, Theorem~\ref{theo3+} is well known for $X=\mathbb R$ and $d(x,y)=|x-y|.$ Consequently we can
find a subsequence $(d(x_{n(k)},a))$ of the sequence $(d(x_n,a))$ such that
$\mathop{\lim}\limits_{k\to\infty}d(x_{n(k)},a)=0$ and $K=\{n(k): k\in\mathbb N\}$ is a
dense subset of $\mathbb N.$ Define the sequence $\tilde y=(y_n)\in\tilde X$ as
\begin{equation*}
y_n:=\begin{cases}
         x_n & \mbox{if} $ $ n\in K \\
         a   & \mbox{if} $ $ n\in\mathbb N \setminus K. \\
         \end{cases}
\end{equation*}
It is easy to see that $\tilde y$ is convergent to $a$ and $\tilde y\asymp \tilde x.$  $ \quad \qquad \qquad \qquad \qquad
\qquad \qquad \square$

\section{Statistical convergence of sequences and their subsequences}

If the given sequence is $d$ -- statistical convergent, it is natural to ask how we can
check that its subsequence is $d$ -- statistical convergent to the same limit.

\begin{theorem}
\label{theo2} Let $(X,d)$ be a metric space,  $\tilde x=(x_n)\in \tilde X$  and let $
\tilde x'=(x_{n(k)})$ be a subsequence of $\tilde x$ such that
\begin{equation}\label{3.1*}
\liminf_{n\to\infty}\frac{\mid K_{\tilde x'}(n)\mid}{n}>0.
\end{equation}
If $\tilde x$ is d-statistical
convergent to $a \in X,$ then $\tilde x'$ is also d-statistical convergent to this a.
\end{theorem}
\begin{proof} Suppose that $(x_n)$ is $d$ -- statistical convergent to $a.$
It is clear that
\begin{equation*}
\left\{ m(k):m(k)\leq n,~~d(x_{m(k)},a)\geq \epsilon \right\}
\subseteq \left\{ m:m\leq n,~~d(x_{m},a)\geq \epsilon \right\}
\end{equation*}%
for all $n.$  Consequently we have
\begin{equation*}\label{equ10}
 \frac{1}{\left\vert K_{\tilde x'}(n)\right \vert}\left\vert \left\{ m(k):m(k)\leq
n,~~d(x_{m(k)},a)\geq \epsilon \right\} \right\vert
\end{equation*}
\begin{equation}\label{equ10}
\leq
\frac{1}{\left\vert K_{\tilde x'}(n)\right \vert}\left\vert \left\{
m:m\leq n,~~d(x_{m},a)\geq \epsilon \right\} \right\vert.
\end{equation}%
The sequence $\tilde x=(x_{m(k)})$ is $d$ -- statistical convergent if, for every
$\epsilon>0,$ we have
$$\limsup_{n\to\infty}\frac{\left\vert\{ m(k): m(k) \le n, d(x_{m(k)},a)\ge \epsilon \}\right\vert}{\left\vert K_{\tilde
x'}(n)\right \vert}=0.$$

Using \eqref{equ10}, we see that the last relation holds if
\begin{equation}\label{equ3}
\limsup_{n\to\infty}\frac{\left\vert \{ m: m \le n, d(x_{m},a)\ge
\epsilon \} \right\vert}{\left\vert K_{\tilde x'}(n)\right \vert}=0.
\end{equation}%
To prove this we can apply the inequality%
\begin{equation}\label{equ4}
\liminf_{n\to\infty}y_n \limsup_{n\to\infty}z_n\le \limsup_{n\to\infty}y_{n}z_{n}
\end{equation}%
which holds for all sequences of nonnegative real numbers with
$0\ne\mathop{\liminf}\limits_{n\to\infty}y_{n}\ne\infty$ (see, for
example, \cite{Dem}). Put
$$y_{n}=\frac{\left\vert K_{\tilde x'}(n)\right \vert}{n} \quad
\mbox{and} \quad z_{n}=\frac{\left\vert\{m: m\le n, d(x_m,a)\ge
\epsilon\}\right\vert}{\left\vert K_{\tilde x'}(n)\right \vert}.$$
Inequality \eqref{3.1*} implies $\mathop{\liminf}\limits_{n\to\infty}y_{n}>0.$ Furthermore it is clear that $\mathop{\liminf}\limits_{n\to\infty}y_{n}\le \mathop{\limsup}\limits_{n\to\infty}y_{n}\le 1.$
Now we obtain
$$y_{n}z_{n}=\frac{\left\vert\{m: m\le n, d(x_{m},a)\ge \epsilon\}\right\vert}{n},$$ so
that
\begin{equation*}
\liminf_{n\to\infty}\frac{\left\vert K_{\tilde x'}(n)\right
\vert}{n}\limsup_{n\to\infty}\frac{\left\vert\{m: m\le n,
d(x_{m},a)\ge \epsilon\}\right\vert}{\left\vert K_{\tilde
x'}(n)\right \vert}
\end{equation*}%
\begin{equation*}
\le\limsup_{n\to\infty}\frac{\left\vert\{m: m\le
n, d(x_{m},a)\ge \epsilon\}\right\vert}{n}.
\end{equation*}
The last inequality implies \eqref{equ3} because \eqref{3.1*} holds and $(x_n)$ is $d$ --
statistical convergent. \end{proof}

\begin{example}\label{ex2} Let $x$ and $y$ be distinct points of a metric space $(X,d)$.
Let us consider the sequence $(x_n),$
\begin{equation*}
x_n:=\begin{cases}
         x & \mbox{if} $\quad $ n $ $ \mbox{is even}\\
         y & \mbox{if} $ \quad$ n $ $ \mbox{is odd}, \\
         \end{cases}
\end{equation*}%
and the subsequences
\begin{equation*}
(x_{2n+1})=(y,y,y,y,y,y,...),~~~ \qquad (x_{2n})=(x,x,x,x,x,x,...).
\end{equation*}%
It is clear that the subsequences $(x_{2n})$ and $(x_{2n+1})$ are
$d$-statistical convergent to $x$ and $y$ respectively. Since $x\neq
y$, Theorem \ref{theo2} implies that $(x_n)$ is not $d$-statistical
convergent.
\end{example}

\begin{theorem}
\medskip \label{theo3} Let $(X,d)$ be a metric space and let $\tilde
x\in \tilde X.$  The following statements are equivalent:
\newline (i) The sequence $\tilde x$  is d-statistical convergent; \newline (ii) Every
subsequence $\tilde x'$ of $\tilde x$ with
$$\liminf_{n\to\infty}\frac{\mid K_{\tilde x'}(n)\mid}{n}>0$$ is d-statistical
convergent; \newline (iii) Every dense subsequence $\tilde x'$ of $\tilde x$ is
d-statistical convergent.
\end{theorem}

\begin{proof}
The implication $(i)\Rightarrow (ii)$ was proved in Theorem~\ref{theo2}. Since every
dense subsequence $\tilde x'$  of $\tilde x$ satisfies the inequality
$$\liminf_{n\to\infty}\frac{\mid K_{\tilde x'}(n)\mid}{n}>0,$$ we have $(ii)\Rightarrow (iii).$
The implication $(iii)\Rightarrow(i)$ holds because $\tilde x$ is a dense subsequence of
itself.\end{proof}

\begin{lemma}\label{lem1} Let $(X,d)$ be a metric space with $\left\vert X\right \vert \ge 2,$
let $\tilde x=(x_n) \in \tilde X$ and let $\tilde x'=(x_{n(k)})$ be
an infinite subsequence of $\tilde x$ such that
\begin{equation}\label{equ13}
\limsup_{n\to\infty}\frac{\mid K_{\tilde x'}(n)\mid}{n}=0.
\end{equation}
There are a sequence $\tilde y \in \tilde X$ and a subsequence $\tilde y'$ of $\tilde y$
such that: $\tilde x \asymp \tilde y$ and $K_{\tilde y'}=K_{\tilde x'}$ and $\tilde y'$
is not d -- statistical convergent.\end{lemma}
\begin{proof} Let $a$ and $b$ be two distinct points of $X.$ Define the sequence $\tilde y=(y_n)\in\tilde X$ by the rule
\begin{equation}\label{y}
y_n:=\begin{cases}
         x_n & \mbox{if} \quad  n\in\mathbb N \setminus K_{\tilde x'}\\
         a & \mbox{if} \quad n=n(k)\in K_{\tilde x'} \, and \, k  \, is \,odd\\
         b & \mbox{if} \quad n=n(k)\in K_{\tilde x'} \, and \, k  \, is \,even.\\
         \end{cases}
\end{equation} The set $\mathbb N \setminus K_{\tilde x'}$ is a statistical dense
subset of $\mathbb N.$ Indeed, the equality
\begin{equation*}
n=\left\vert\{m\in K_{\tilde x'}: m\le n\}\right\vert + \left\vert\{m\in \mathbb N
\setminus K_{\tilde x'}: m\le n\}\right\vert
\end{equation*} holds for each $n\in\mathbb N.$ It implies the inequalities
\begin{equation*}\label{equ15}
\liminf_{n\to\infty}\frac{\left\vert\{m\in \mathbb N\setminus K_{\tilde x'}: m\le
n\}\right\vert}{n}=\liminf_{n\to\infty}\left( 1-\frac{\left\vert\{m\in K_{\tilde x'}: m\le
n\}\right\vert}{n}\right)
\end{equation*}

\begin{equation*}\label{equ15}
= 1-\limsup_{n\to\infty}\frac{\left\vert\{m\in K_{\tilde x'}: m\le
n\}\right\vert}{n}.
\end{equation*}
Using \eqref{equ13} we obtain
\begin{equation*}
1 = \liminf_{n\to\infty}\frac{\left\vert\{m\in \mathbb N\setminus K_{\tilde x'}: m\le
n\}\right\vert}{n}\le \limsup_{n\to\infty}\frac{\left\vert\{m\in \mathbb N\setminus
K_{\tilde x'}: m\le n\}\right\vert}{n}\le 1.
\end{equation*} Consequently
\begin{equation*}
\lim_{n\to\infty} \frac{\left\vert\{m\in \mathbb N\setminus K_{\tilde x'}: m\le
n\}\right\vert}{n}=1.
\end{equation*}
Thus $\tilde x \asymp \tilde y.$
Define the desired subsequence $\tilde y'$ of $\tilde y$ as $\tilde
y'=(y_{n(k)}),$ (see \eqref{y}). As in Example~\ref{ex2} we see that
$\tilde y'$ is not $d$ -- statistical convergent.
 \end{proof}

\begin{lemma}\label{lem2}Let $(X,d)$ be a metric space, $a\in X,$ $\tilde x$ and $\tilde y$ belong to $\tilde X$ and let $\tilde x$ be a d -- statistical convergent to $a$ sequence. If $\tilde x \asymp \tilde y,$ then $\tilde y$ is also d -- statistical convergent to $a.$\end{lemma}
\begin{proof}
Suppose that $\tilde y \asymp \tilde x.$ Define a subset $M$ of the set $\mathbb N$ as
$$(n\in M) \Leftrightarrow (x_n \ne y_n).$$ Then, by Definition~\ref{def6}, $\mathbb N \setminus M$ is statistical dense. It implies
the equality
\begin{equation}\label{equ17}
\lim_{n\to\infty}\frac{\left\vert\{m\in M: m\le n\}\right\vert}{n}=0.
\end{equation}
Let $\epsilon$ be a strictly positive number. It follows directly from the definition of
the set $M$ that the inclusion
\begin{equation}
\{m\in\mathbb N: m\le n \, and \, d(y_m,a) \ge \epsilon \} \end{equation}
\begin{equation*}
\subseteq \{m\in M : m\le n \}\cup \{m\in \mathbb N : m\le n  \, and \, d(x_m, a)\ge
\epsilon\}
\end{equation*} holds for each $n\in\mathbb N.$ Using this inclusion and equality
\eqref{equ17} we obtain
$$
\limsup_{n\to\infty}\frac{\left\vert \{m\in\mathbb N: m\le n \, and \, d(y_m,a) \ge
\epsilon \}\right \vert }{n}$$ $$\le \limsup_{n\to\infty}\frac{\left\vert\{m\in M: m\le
n \} \right \vert}{n}+\limsup_{n\to\infty}\frac{\left\vert\{m\in \mathbb N : m\le n \,
and \, d(x_m, a)\ge \epsilon\}\right \vert }{n}
$$
$$=\limsup_{n\to\infty}\frac{\left\vert\{m\in \mathbb N : m\le n
\, and \, d(x_m, a)\ge \epsilon\}\right \vert }{n}.$$ Since $\tilde x$ is $d$ --
statistical convergent to $a,$ we have
$$\limsup_{n\to\infty}\frac{\left\vert\{m\in \mathbb N : m\le n
\, and \, d(x_m, a)\ge \epsilon\}\right \vert }{n}=0$$ for every $\epsilon >0.$
Consequently the inequality
\begin{equation}\label{equ19}
\limsup_{n\to\infty}\frac{\left\vert \{m\in\mathbb N: m\le n \, and \, d(y_m,a) \ge
\epsilon \}\right \vert }{n}\le 0
\end{equation}
holds for every $\epsilon >0.$ Using \eqref{equ19} we obtain
\begin{equation*}
0\le \liminf_{n\to\infty}\frac{\left\vert \{m\in\mathbb N: m\le n \, and \, d(y_m,a) \ge
\epsilon \}\right \vert }{n}
\end{equation*}
\begin{equation*}
\le \limsup_{n\to\infty}\frac{\left\vert \{m\in\mathbb N: m\le n \,
and \, d(y_m,a) \ge \epsilon \}\right \vert }{n}\le 0.
\end{equation*}
Hence the limit relation \begin{equation*} \lim_{n\to\infty}\frac{\left\vert
\{m\in\mathbb N: m\le n \, and \, d(y_m,a) \ge \epsilon \}\right \vert }{n}=0
\end{equation*} holds. It still remains to note that the last limit
relation holds for every $\epsilon >0$ if and only if $\tilde y$ is $d$ -- statistical
convergent to $a.$ \end{proof}
\begin{theorem}\label{theo5}
Let $(X,d)$ be a metric space with $\left\vert X \right\vert \ge 2,$ $a\in X,$ and let
$\tilde x \in \tilde X$ be d -- statistical convergent to $a$. Then for every infinite
subsequence $\tilde x'$ of $\tilde x$ with
\begin{equation*}
\limsup_{n\to\infty}\frac{\mid K_{\tilde x'}(n)\mid}{n}=0
\end{equation*} there are a sequence $\tilde y \in \tilde X$ and a subsequence $\tilde
y'$ of $\tilde y$ such that: \newline (i) $\tilde y \asymp \tilde x$ and $K_{\tilde x'}
= K_{\tilde y'};$ \newline (ii) $\tilde y$ is d -- statistical convergent to $a;$
\newline (iii) $\tilde y'$ is not d -- statistical convergent.
\end{theorem}
\begin{proof}
By Lemma~\ref{lem1} there are $\tilde y$ and $\tilde y'$ such that $(i)$ and $(iii)$
hold. To prove $(ii)$ note that $(i)\Rightarrow (\tilde y \asymp \tilde x)$ and $\tilde x$ is a
$d$ -- statistical convergent to a sequence. Consequently, by Lemma~\ref{lem2}, $\tilde
y$ is also $d$ -- statistical convergent to $a.$
\end{proof}

Using this theorem we obtain the following ``weak'' converse of Theorem~\ref{theo2}.
\begin{theorem}
Let $(X,d)$ be a metric space with $\left\vert X \right\vert \ge 2$ and let $\tilde x\in
\tilde X$ be a $d$~--~statistical convergent sequence. Assume $\tilde x'$ is a
subsequence of $ $ $\tilde x$ having the following property: if $\tilde y \asymp \tilde
x$ and $\tilde y'$ is a subsequence of $\tilde y$ such that $K_{\tilde x'}=K_{\tilde
y'},$ then $\tilde y'$ is d -- statistical convergent. Then the inequality
\begin{equation}\label{equ20}
\limsup_{n\to\infty}\frac{\left\vert K_{\tilde x'}(n)\right\vert}{n} >0
\end{equation} holds.
\end{theorem}
\begin{proof}
For $\tilde x'$ we have either \eqref{equ20} or
\begin{equation*}
\limsup_{n\to\infty}\frac{\left\vert K_{\tilde x'}(n)\right\vert}{n} =0.
\end{equation*} If the last equality holds, then by Theorem~\ref{theo5} there are $\tilde y$ and $\tilde y'$ such that $\tilde y\asymp\tilde x,$ $K_{\tilde x'}=K_{\tilde y'}$ and $\tilde y'$ is not $d$ --
statistical convergent. It contradicts the assumption of the theorem. \end{proof}
Similarly we have a ``weak'' converse of Theorem~\ref{theo5}.
\begin{theorem}\label{theo7}
Let $(X,d)$ be a metric space, $a\in X,$ and let $\tilde x \in \tilde X$ be a d --
statistical convergent to a sequence. Suppose $\tilde x' = (x_{n(k)})$ is a subsequence
of $\tilde x$ for which there are $\tilde y \in \tilde X$ and $\tilde y'$ such that
conditions $(i)$ and $(iii)$ of Theorem~\ref{theo5} hold. Then we have the equality
\begin{equation}\label{equ21}
\liminf_{n\to\infty}\frac{\left\vert K_{\tilde x'}(n)\right\vert}{n} =0.
\end{equation}
\end{theorem}
To prove this result we use the next lemma.

\begin{lemma}\label{lem3}
Let $(X,d)$ be a metric space, $\tilde x$ and $\tilde y$ belong to $\tilde X$ and let
$\tilde x \asymp \tilde y.$ If $K$ is a subset of $\mathbb N$ such that
\begin{equation}\label{equ22}
\liminf_{n\to\infty}\frac{\left\vert K(n)\right\vert}{n} >0
\end{equation} and if $\tilde x' = (x_{n(k)})$ and $\tilde y' = (y_{n(k)})$ are subsequences
of $\tilde x$ and, respectively, of $\tilde y$ such that $K_{\tilde x'}=K_{\tilde
y'}=K,$ then the relation $\tilde y' \asymp \tilde x'$ holds.
\end{lemma}
\begin{proof}
It is sufficient to show that
\begin{equation}\label{equ23}
\limsup_{m\to\infty}\frac{\left\vert\{n(k) \in K: x_{n(k)}\ne y_{n(k)} \, and \, n(k)\le
m\}\right\vert}{\left\vert K(m)\right\vert}=0.
\end{equation} Since the inclusion
\begin{equation*}
\{n(k) \in K: x_{n(k)}\ne y_{n(k)} \, and \, n(k)\le m\}\subseteq\{n \in \mathbb N:
x_{n}\ne y_{n} \, and \, n\le m\}
\end{equation*}holds for every $m\in\mathbb N,$ we have
\begin{equation}\label{equ24}
\limsup_{m\to\infty}\frac{\left\vert\{n(k) \in K: x_{n(k)}\ne y_{n(k)} \, and \, n(k)\le
m\}\right\vert}{\left\vert K(m)\right\vert} \end{equation}
\begin{equation*}\le\limsup_{m\to\infty}\frac{\left\vert\{n\in\mathbb N:
x_n \ne y_n \, and \, n\le m\} \right\vert}{{\left\vert K(m)\right\vert}}
\end{equation*}
\begin{equation*}
\le\limsup_{m\to\infty}\frac{m}{\left\vert
K(m)\right\vert}\limsup_{m\to\infty}\frac{\left\vert\{n\in\mathbb N: x_n \ne y_n \, and
\, n\le m\} \right\vert}{m}
\end{equation*}
\begin{equation*}
=\frac{\mathop{\limsup}\limits_{m\to\infty}\frac{\left\vert\{n\in\mathbb N:\, x_n \ne y_n
\, and \, n\le m\} \right\vert}{m}}{\mathop{\liminf}\limits_{m\to\infty}\frac{\left\vert
K(m) \right \vert}{m}}.
\end{equation*} Inequality \eqref{equ22} implies that
\begin{equation}\label{equ25}
0 \le\frac{1} {\mathop{\liminf}\limits_{m\to\infty}\frac{\left\vert K(m) \right
\vert}{m}}<+\infty.
\end{equation}Moreover we have
\begin{equation*}
\limsup_{m\to\infty}\frac{\left\vert\{n\in\mathbb N: x_n \ne y_n \, and \, n\le m\}
\right\vert}{m}=0
\end{equation*} because $\tilde x\asymp\tilde y.$
Now \eqref{equ23} follows from the last equality, \eqref{equ24} and \eqref{equ25}.
\end{proof}

\emph{Proof of Theorem~\ref{theo7}.} For $\tilde x'$ we have either \eqref{equ21} or
\begin{equation}\label{equ26}
\liminf_{n\to\infty}\frac{\left\vert K_{\tilde x'}(n) \right\vert}{n}>0.
\end{equation}
It suffices to show that the last inequality contradicts the conditions of
Theorem~\ref{theo7}. Let $\tilde y\in\tilde X$ and $\tilde y'$ be a sequence and its
subsequence such that conditions $(i)$ and $(iii)$ of Theorem~\ref{theo5} hold. By
condition $(i)$ we have $K_{\tilde x'}=K_{\tilde y'}$ and
$\tilde x\asymp\tilde y.$ \, Now \newline using \eqref{equ26} and Lemma~\ref{lem3}, we obtain  that $\tilde x'\asymp\tilde y'.$ Moreover, applying Theorem~\ref{theo2}, we see that $\tilde
x'$ is $d$ -- statistical convergent to $a.$ Since $\tilde x'\asymp\tilde y',$
Lemma~\ref{lem2} shows that $\tilde y'$ is also $d$ -- statistical convergent to $a,$
contrary to condition $(iii)$ of \, Theorem~\ref{theo5}.$ \qquad \qquad
\qquad \qquad \qquad \qquad \qquad \qquad\qquad \qquad\qquad \qquad\qquad \qquad\qquad \qquad\Box$

\end{document}